\documentclass{amsart}

\usepackage{amssymb}

\usepackage{hyperref}
\newcommand{\arxiv}[1]{\href{http://www.arXiv.org/abs/#1}{arXiv:#1}}

\newtheorem{theorem}{Theorem}[section]
\newtheorem{lemma}[theorem]{Lemma}

\theoremstyle{definition}

\theoremstyle{remark}

\numberwithin{equation}{section}

\newcommand{\cset}{\mathbf{C}}

\newcommand{\conv}{\operatorname{conv}}

\newcommand{\cK}{{\mathcal K}}

\newcommand{\cS}{{\mathcal S}}

\def\R{\mathbf R}

\def\Rp{\R_+}
\def\Rpn{\Rp^n}
\def\Rpnn{\Rp^{n\times n}}

\begin{document}


\title[On common eigenvectors for semigroups]{On common eigenvectors for semigroups of matrices in tropical and traditional mathematics}

\thanks{This work is supported by the RFBR
grant 12--01--00886-a, the joint RFBR-CNRS grant 11--01--93106-a, and the EPSRC grant RRAH15735}


\author{Grigory~B.~Shpiz}

\address{Grigory~B.~Shpiz, 
Scientific Institute for Nuclear Physics in the Moscow State University,
Moscow, Russia}

\email{islc@dol.ru}

\author{Grigory~L.~Litvinov}

\address{Grigory~L.~Litvinov, The A.~A.~Kharkevich Institute for Information Transmission Problems RAS and the Poncelet Laboratory, Moscow, Russia}

\email{glitvinov@gmail.com}

\author{Serge\u{\i}~N.~Sergeev}

\address{Serge\u{\i}~N.~Sergeev, Moscow Center for Continuous Mathematical Education, Russia, and
University of Birmingham, Edgbaston B15 2TT, UK}

\email{sergiej@gmail.com}


\begin{abstract}
We prove the existence of a common eigenvector for commutative, nilpotent and quasinilpotent semigroups of matrices
with complex or real nonnegative entries both in the conventional and tropical linear algebra.
\end{abstract}

\maketitle

\vspace{1cm}

\normalsize

\section{Preliminaries}

We consider semigroups of matrices with entries in 1) the complex field $\cset$, 2) the semifield of nonnegative real
numbers with usual arithmetics, 3) the semifield of nonnegative numbers $\Rp(\max)$ with idempotent addition $a\oplus b:=\max(a,b)$ and the usual multiplication $ab:=a\times b$.

The latter example is known as the {\bf max-times semifield}; this semifield is isomorphic to the so-called tropical algebra or max-plus algebra. This case is one of our main motivations. As in the case of usual arithmetics, these semifields naturally extend
to matrices and vectors, giving rise to {\bf max-linear algebra}~\cite{But:10}, and
to the tropical convex geometry~\cite{GM-09} of max-algebraic subspaces of the nonnegative orthant $\Rp^n$. These subspaces are subsets of $\Rp^n$ closed under taking componentwise maximum of vectors, and multiplication by a nonnegative scalar. They are
also known as idempotent linear spaces over $\Rp(\max)$\cite{LMS-01,LMS-02}, or max cones~\cite{BSS}.

We give a general proof that in all these cases, any semigroup of pairwise commuting matrices has a common eigenvector. This can be seen as an extension of a recent result of Katz, Schneider, Sergeev~\cite{KSS}. We also
extend this result on existence of a common eigenvector to the case of nilpotent and quasinilpotent semigroup, in full analogy with nilpotent
groups. The notion of quasinilpotent semigroup was suggested by G.B. Shpiz.
See, e.g.,~\cite{IJK,LMS-02,Merlet}
for other recent results on the existence of a common eigenvector of  tropical matrix groups and semigroups.

In the sequel, the cases described above are considered simultaneously. More precisely, we will consider linear operators in spaces (semimodules)
$\cK^n=\underbrace{\cK\times\ldots\times\cK}_{n}$, where $\cK$ is one of
the following semifields: 1) the complex field $\cset$, or 2) the semifield of nonnegative numbers $\Rp$ with the usual arithmetics, 3) the max-times semifield
$\Rp(\max)$.

We will also consider {\bf subspaces} of $\cK^n$.
In the first case, subspaces of $\cK^n$ are just linear subspaces of $\cset^n$ in the usual meaning of the word. In the second case, subspaces of $\cK^n$ are convex cones in
$\Rp^n$; in the third case, they are the max cones~\cite{BSS} mentioned above.
In any of these cases, a subspace is called
{\bf nontrivial} if it contains a nonzero vector.
It will be called {\bf closed}
if it is closed in the sense of the usual Euclidean topology. A nontrivial subspace $W\subseteq\cK^n$ is called {\bf invariant} under the action of a linear
operator $A$, if $AW\subseteq W$, that is, if $Ax\in W$ for all $x\in W$. 

The eigenvector existence theorems (Lemma~\ref{l:eigenvector}) 
will be an important ingredient in the proofs below.
Note that the spectral theory of max-linear maps is well-known only for the case of nonnegative orthant, while we will need the eigenvector existence in an arbitrary invariant max cone. In the classical nonnegative algebra, such results follow either from the Brouwer fixed point theorem, or by extending Wielandt's proof of the Perron-Frobenius theorem to cones. 
In Section~\ref{s:shpiz} we show how these two techniques can be recast in 
the max-algebraic setting. Let us remark that these techniques have been developed in
the framework of more general tropical spaces and over idempotent semirings more
general than max-times in~\cite{LMS-01,LMS-02,Shpiz,Schauder}.

Proof of the following Lemma is elementary, but
it is recalled for the reader's convenience.

\begin{lemma}
\label{l:inters-nontriv}
Let $\{W_n\}$ be a sequence of closed nontrivial subspaces of $\cK^n$,
where $\cK\in\{\cset,\Rp,\Rp(\max)\}$, such that
$W_1\supseteq W_2\supseteq W_3\supseteq\ldots$. Then the intersection
$\cap_{i=1}^{\infty} W_i$ is a closed nontrivial space.
\end{lemma}
\begin{proof}
For each $i$, consider intersection of $W_i$ with the sphere
$\{y\colon ||y||=1\}$ where $||\cdot||$ is, for instance, the usual Euclidean
$2$-norm. Denote this intersection by $S_i$. Evidently $0\notin S_i$ and
all $S_i$ are non-empty and compact and satisfy $S_1\supseteq S_2\supseteq S_3\supseteq\ldots$, hence $\cap_{i=1}^{\infty} S_i$ is a compact
non-empty set and $\cap_{i=1}^{\infty} W_i$ is a closed nontrivial space.
\end{proof}

\begin{lemma}
\label{l:eigenvector}
Let $W$ be a closed subspace of $\cK^n$, invariant under the action of a linear operator
$A$, where $\cK\in\{\cset,\Rp,\Rp(\max)\}$.  Then $W$ contains an eigenvector of $A$.
\end{lemma}
\begin{proof}
The case of $\Rp(\max)$ will be treated in Section~\ref{s:shpiz}.

In the case when $A$ is a complex-valued matrix and $W$ is a
subspace in $\cset^n$, this result is standard. In the case when $A$
is a nonnegative matrix and $W$ is a closed invariant cone of $A$ in
$\Rpn$, the result follows by a standard application of the Brouwer
fixed point theorem. A more precise formulation is known as the
finite-dimensional Kre\u{\i}n-Rutman theorem or the Perron-Frobenius
theorem for cones, see Schneider~\cite{Sch-65} Theorem~0. 
\end{proof}

We conclude with the following facts, which are evident in all algebraic
structures that we
consider. Recall that for $\lambda_A\in\cK$, the set
$\{v\in\cK^n\colon Av=\lambda_Av\}$ is called
the eigenspace of $A$ associated with $\lambda_A$. Linear operators $A$ and $B$
commute if $AB=BA$, that is, if $A(Bv)=B(Av)$ for all $v\in\cK^n$.

\begin{lemma}
\label{l:invariant}
Let $W\subseteq \cK^n$ be a nonzero eigenspace of a linear operator $A$ associated with some
eigenvalue of $A$, where $\cK\in\{\cset,\Rp,\Rp(\max)\}$.
\begin{itemize}
\item[1.] $W$ is closed;
\item[2.] Let $B$ be a linear operator that commutes with $A$. Then
$BW\subseteq W$.  
\end{itemize}
\end{lemma}
\begin{proof}
The first part follows from the continuity of linear operators in all spaces 
$\cK^n$ that we consider. For the second part, let $v\in\cK^n$ satisfy
$Av=\lambda_Av$ for some $\lambda_A\in\cK$. 
Then $A(Bv)=B(Av)=B(\lambda_A v)=\lambda_A (Bv)$, thus the eigenspace of $A$ associated
with $A$ is invariant under $B$.
\end{proof}

\if{
\begin{lemma}
\label{l:invariant}
Let $V\subseteq\cK^n$, where $\cK\in\{\cset,\Rp,\Rp(\max)\}$, be the intersection of the eigenspace
of a linear operator $A$, associated with some
eigenvalue, with a closed subspace $W$ of $\cK^n$ invariant under the
action of $A$. Then
\begin{itemize}
\item[1.] $V$ is closed;
\item[2.] Let $B$ be a linear operator that commutes with $A$ and $W$ is invariant under $B$, i.e. $BW\subseteq W$. Then
$V$ is also invariant under $B$.
\end{itemize}
\end{lemma}
}\fi

\section{Main results}
\label{s:main}

We say that a semigroup $\cS$ of linear operators acting in a subspace $V\subseteq\cK^n$ has
an eigenvector $v$, if for each $A\in\cS$ there is a value
$\lambda_A\in\cK$ such that $Av=\lambda_A v$, and $v\in V$ is nonzero. We say that
$\cS$ is commutative if $AB=BA$ for all $A,B\in\cS$.

\begin{theorem}
\label{t:comm-eig}
Let $\cS$ be a commutative semigroup of linear operators leaving a closed subspace $V\subseteq\cK^n$ invariant, where $\cK\in\{\cset,\Rp,\Rp(\max)\}$. Then
$\cS$ has an eigenvector $v\in V$.
\end{theorem}
\begin{proof}
Let us consider the collection of all closed nontrivial subspaces of $V$ invariant
under all operators of $\cS$. Such subspaces will be called invariant by abuse of terminology. This collection contains $V$, so it is non-empty.
For any $\cK\in\{\cset,\Rp,\Rp(\max)\}$, the intersection of any chain of closed invariant
spaces is a closed invariant space, and it is nontrivial by
Lemma~\ref{l:inters-nontriv}. Then we can apply Zorn's Lemma to get a nontrivial minimal (under inclusion) closed invariant space $W$. Consider a subspace of $W$ consisting of all
eigenvectors of a matrix $A\in\cS$ associated with an eigenvalue $\lambda_A$.   Denote this
subspace by $W'$. By Lemma~\ref{l:eigenvector} $W'$ contains a nonzero vector. Observe that $W'$ is the intersection of $W$ with the eigenspace
of $A$ associated with $\lambda_A$. By Lemma~\ref{l:invariant}, $W'$ is a closed invariant
subspace of $W$, and the minimality implies $W=W'$. As we took
an arbitrary $A\in\cS$, it follows that for all $A\in\cS$ the subspace $W$ contains (besides the zero vector) only eigenvectors of $A$ associated with some eigenvalue $\lambda_A$.
All these eugenvectors are common eigenvectors of all matrices in $\cS$.
\end{proof}

For a semigroup $\cS$, define its subsemigroup $\cS^{(k)}$ consisting of all
$A\in\cS$ that can be represented as $B_1\cdot\ldots\cdot B_k$ with
$B_1,\ldots,B_k\in\cS$. Indeed, this is a subsemigroup since
$B_1\cdot\ldots\cdot B_{k}C_1\cdot\ldots\cdot C_{k}$ can be represented, for instance,
as $(B_1\cdot\ldots\cdot B_{k}C_1)\cdot\ldots\cdot C_{k}$.
For a similar reason, we have $\cS^{(k)}\subseteq\cS^{(k-1)}$ for all $k>1$.
$\cS$ is called {\it quasinilpotent} if $\cS^{(k)}$ is
a commutative semigroup for some $k$. Recall that a semigroup $\cS$ with a zero element $0$
is {\it nilpotent} if $\cS^{(k)}$ = $\{0\}$ for some $k$. Of course, every commutative or nilpotent
semigroup is quasinilpotent.

\begin{theorem}
\label{t:quas-eig}
Let $\cS$ be a quasinilpotent (or nilpotent) semigroup of linear operators leaving a closed subspace $V\subseteq\cK^n$ invariant, where $\cK\in\{\cset,\Rp,\Rp(\max)\}$. Then
$\cS$ has an eigenvector $v\in V$.
\end{theorem}
\begin{proof}
Let $\cS^{(t)}$, for some $t\geq 1$, be a commutative semigroup (or let
$\cS^{(t)}=\{0\}$ in the nilpotent case). 
By Theorem~\ref{t:comm-eig}, $\cS^{(t)}$ has an eigenvector.
It suffices to prove that for $k>1$, if $\cS^{(k)}$ has an eigenvector then so does $\cS^{(k-1)}$, and then a straightforward induction can be applied.

Let $u$ be an eigenvector of $\cS^{(k)}$ and suppose that there exists $A\in\cS^{(k)}$,
for which this eigenvector is associated with a nonzero eigenvalue $\lambda$.
For each $B\in\cS^{(k-1)}$, we have
$BA\in\cS^{(k)}$ (as $A\in\cS$), and hence $BAu=\mu u$ for some $\mu$. Substituting $Au=\lambda u$
we obtain $Bu=\lambda^{-1}\mu u$, so $u$ is also an eigenvector of $\cS^{(k-1)}$.

Otherwise, suppose that there is $u\in V$, $u\neq 0$, such that $Au$ is zero for all
$A\in\cS^{(k)}$. If we also have $Bu$ zero for all $B\in\cS^{(k-1)}$ then $u$ is an
eigenvector of $\cS^{(k-1)}$ (with a common eigenvalue equal to zero). Otherwise, take
$B'\in\cS^{(k-1)}$ such that $B'u$ is nonzero, but then $BB'u$ is still zero for all 
$B\in\cS^{(k-1)}$ since $BB'\in\cS^{(k)}$.
So $B'u$ is an eigenvector of $\cS^{(k-1)}$ (with a common eigenvalue equal to zero).

We proved the claim in all possible cases.
\end{proof}

\noindent {\bf Remark.}
It is well known that a connected nilpotent or solvable complex matrix Lie group has a common
eigenvector (the Lie and Engel theorems, see, e.g. \cite{Serre}). For abstract nilpotent (but not for solvable)
groups of linear continuous operators in a tropical linear Archimedean space the corresponding result
was proved in \cite{LMS-02}.

\section{Proofs of Lemma~\ref{l:eigenvector} for max cones}
\label{s:shpiz}

In this section we give two proofs of Lemma~\ref{l:eigenvector} for max cones in
$\Rpn$. The first proof is based on Brouwer's fixed point theorem following the
idea of the proof of~\cite{Schauder} Theorem~1, and the second
proof can be seen as a max-algebraic analogue of Wielandt's proof
of the Perron-Frobenius theorem being an adaptation of the argument of Shpiz~\cite{Shpiz}
Theorem~3. 

\subsection{Application of Brouwer's fixed-point theorem}
\label{ss:brouwer-proof}

Let $W$ be a closed max cone in $\Rpn$ and let 
\begin{equation}
\label{sdef}
S=W\cap \{x\mid \sum_{i=1}^n x_i=1\}.
\end{equation}
Here, $\sum$ is the ordinary sum (not to confuse with 
the componentwise maximum). Denote by $\conv(S)$ the convex hull of
$S$, i.e., the smallest conventionally convex set containing $S$. As it is known from 
convex analysis, this is the same as the set of all convex combinations 
$\sum_{i=1}^m \mu_i x^i$, where $x^1,\ldots,x^m\in S$, all $\mu_i$ are nonnegative and satisfy $\sum_{i=1}^m \mu_i=1$. Since $W$ is closed, $S$ is compact, and $\conv(S)$ is a compact convex subset of $\{x\mid \sum_{i=1}^n x_i=1\}$.

A nonlinear projector $P_W\colon\Rpn\to W$ (see~\cite{CGQS,LMS-01}) can be defined by
\begin{equation}
\label{e:proj-def}
P_W(y):=\sup\{x\in W\mid x\leq y\},
\end{equation}
where $\sup$ denotes the componentwise supremum. As $W$ is closed, this supremum is reached.
In particular, $P_W(x)=x$ for all $x\in W$.

\begin{lemma}
\label{l:proj}
$P_W$ is a continuous operator on $\Rpn$ and 
$P_W(z)\neq 0$ for any nonzero $z\in\conv(S)$.
\end{lemma}
\begin{proof}
A proof of continuity of $P_W$ 
can be found, for instance, in~\cite{CGQS} Theorem~3.11. To prove the second part of the claim, it is sufficient to find,
for each nonzero $z\in\conv(S)$, a vector $y\in W$ with $0\neq y\leq z$. For this, recall that
$z$ can be represented as $z=\sum_{i=1}^m \mu_i x^i$ for some $x^1,\ldots,x^m\in W$, where all
$\mu_i$ are nonnegative and satisfy $\sum_{i=1}^m \mu_i=1$. Take for $y$ the componentwise maximum of $\mu_1 x^1,\ldots,\mu_m x^m$, then $y\leq z$, $y\neq 0$ and $y\in W$.   
\end{proof}

\medskip\noindent{\bf Proof of Lemma~\ref{l:eigenvector}} 
We can assume that $A\otimes x\neq 0$ for all $x\in W$, otherwise there exists an eigenvector
of $A$ associated with the zero eigenvalue.

Consider the set $S$~\eqref{sdef}, its conventional convex hull $\conv(S)$,
the mappings $\pi\colon\conv(S)\to S$ and $\gamma\colon S\to S$ defined by 
\begin{equation}
\label{e:pigdef}
\pi(x):=(P_W x)/\sum_{i=1}^n (P_W x)_i\qquad
\gamma(x):=(A\otimes x)/\sum_{i=1}^n (A\otimes x)_i. 
\end{equation}
By the continuity of $A$ and since $A\otimes x\neq 0$ on $S$,  $\gamma$ is a continuous mapping of $S$ in itself. By the continuity of $P_W$ and since $P_W x\neq 0$ on 
$\conv(S)$ (see Lemma~\ref{l:proj}), $\pi$ is a continuous mapping of $\conv(S)$ in itself.
Also note that $\pi(x)=x$ for $x\in S$.
The composition
$\gamma\pi$ is continuous on $\conv(S)$. Applying Brouwer's fixed point theorem
to $\gamma\pi$ acting on $\conv(S)$ (which is a compact convex set) we obtain a point
$x$ satisfying $\gamma\pi(x)=x$. But then we have $x\in S$, since $\gamma\pi(x)\in S$. Since $\pi(x)=x$ for $x\in S$, we obtain $\gamma(x)=x$, so $x\in S\subseteq W$
is a nonzero eigenvector of $A$.

\subsection{Algebraic proof}
\label{ss:alg-proof}

We adapt a proof of the eigenvector existence theorem of Shpiz~\cite{Shpiz}.

The max-algebraic matrix powers of $A\in\Rpnn$ will be denoted by $A^{\otimes t}:=\underbrace{A\otimes\ldots\otimes A}_{t}$. We have $A^{\otimes 0}=I$, the identity matrix. We recall that the max-algebraic matrix multiplication is homogeneous: $A\otimes (rx)=rA\otimes x$ for all $x\in\Rpn$ and $r\in\Rp$, monotone: $x\leq y\Rightarrow A\otimes x\leq A\otimes y$ for all $x,y\in\Rpn$, and continuous.
By the continuity of max-algebraic matrix multiplication, the
extended max-linearity
\begin{equation}
\label{e:maxlinearity}
A\otimes \bigoplus_{\mu\in S} v^{\mu} =\bigoplus_{\mu\in S} A\otimes v^{\mu}
\end{equation}
holds for $A\in\Rpnn$ and any set $\{v^{\mu}\}_{\mu\in S}\subseteq\Rpn$ bounded from above.
The $\bigoplus$ sign denotes the componentwise supremum of a bounded (but possibly infinite) subset of $\Rpn$.

We will use the notation
\begin{equation}
\label{e:pseudodiv}
[v:w] =\max\{\lambda\mid \lambda w\leq v\}=\min_{i\colon w_i\neq 0} v_iw_i^{-1}.
\end{equation}
for $v,w\in\Rpn$, where $w\neq 0$. Note that $[v:w]\cdot w\leq v$.

The proof of Lemma~\ref{l:eigenvector} is preceded by the following technical result.

\begin{lemma}[Shpiz~\cite{Shpiz}, Lemma~2]
\label{l:shpiz}
Let $W$ be a nonzero minimal (under inclusion) closed max cone in $\Rpn$ invariant under the
max-algebraic multiplication by a matrix $A\in\Rpn$. Assume that
$A^{\otimes t} v\neq 0$ for any nonzero $v\in W$ and any $t\geq 1$.
For every integer $t\geq 1$,
\begin{itemize}
\item[(i)] $[v:A^{\otimes t}\otimes v]$ is the same for all
$v\in W$, $v\neq 0$.
\item[(ii)] $[v:A^{\otimes t}\otimes v]=[v:A\otimes v]^t$ holds for all $v\in W$, $v\neq 0$.
\end{itemize}
\end{lemma}
\begin{proof}
(i): For each $v\in W$, $v\neq 0$,
\begin{equation}
W_{v,t}:=\{w\mid [v:A^{\otimes t}\otimes v]\cdot A^{\otimes t}\otimes w\leq w\}
\end{equation}
is a closed subspace invariant under $A$. It is nontrivial (i.e., not reduced to the
zero vector) since $v\in W_{v,t}$. By the minimality of $W$ we obtain $W_{v,t}=W$, and
using~\eqref{e:pseudodiv} this implies that $[v:A^{\otimes t}\otimes v]\leq
[w:A^{\otimes t}\otimes w]$ for all nonzero $v,w\in W$. Swapping $v$ and $w$ we get the reverse
inequality, which proves the claim.

(ii): Fix a nonzero $v\in W$,
and denote $r:=[v:A^{\otimes t}\otimes v]^{1/t}$.
It will be shown that $[v:A\otimes v]=r$.  The inequality $[v:A\otimes v]\leq r$ is relatively
easy: iterating $[v:A\otimes v]\cdot A\otimes v\leq v$ one obtains that
$[v:A\otimes v]^t\cdot A^{\otimes t}\otimes v\leq v$, and
this implies $[v:A\otimes v]^t\leq [v:A^{\otimes t}\otimes v]$.

For the reverse inequality, consider the vector $w:=\bigoplus_{\ell\geq 0} r^\ell A^{\otimes \ell}\otimes v$. It will be shown a bit later that the sequence of vectors $\{r^tA^{\otimes \ell}\otimes v\}_{\ell\geq 0}$ is bounded, so its supremum is finite. Applying $rA$ to $w$ and using the extended
max-linearity of $A$~\eqref{e:maxlinearity} we get
$$
rA\otimes w=\bigoplus_{\ell\geq 1} r^\ell A^{\otimes \ell}\otimes v\leq w.
$$
It follows that $r\leq [w:A\otimes w]$, and since $[w:A\otimes w]=[v:A\otimes v]$
by part (i), we obtain $r\leq [v:A\otimes v]$ and thus $r=[v:A\otimes v]$.

To complete the proof, it remains to show that the sequence $\{r^{\ell}A^{\otimes \ell}\otimes v\}_{\ell\geq 0}$ is bounded. For this, it is observed, by representing
$\ell=tp+q$ with $p\geq 0$ and $0\leq q<t$, by repeatedly applying
$[v:A^{\otimes t}\otimes v] A^{\otimes t}\otimes v\leq v$ and using the homogeneity
and monotonicity of max-algebraic matrix multiplication, that
$$
r^\ell A^{\otimes\ell}\otimes v=[v:A^{\otimes t}\otimes v]^{p+(q/t)} A^{\otimes tp+q}\otimes v
\leq [v:A\otimes v]^{q/t} A^{\otimes q}\otimes v.
$$
Thus $r^\ell A^{\otimes\ell}\otimes v$ can be bounded by
the componentwise maximum of all vectors $[v:A\otimes v]^{q/t}\cdot A^{\otimes q}\otimes v$, over $q$ from $0\leq q<t$.
\end{proof}

\medskip\noindent{\bf Proof of Lemma~\ref{l:eigenvector}}~(\cite{Shpiz}~Theorem~3):

If $A^{\otimes t}\otimes v=0$ for some nonzero $v\in W$ and $t\geq 1$ then $A$ has an eigenvector with zero eigenvalue.

Assume that
$A^{\otimes t} v\neq 0$ for any nonzero $v\in W$ and any $t\geq 1$.
We can also assume without loss of generality that $W$ is a minimal (under inclusion) closed max cone invariant under the max-algebraic multiplication by $A$, otherwise we can use Lemma~\ref{l:inters-nontriv} and select such a max cone by means of Zorn's Lemma.
Consider the set $C=\{x\in W\colon \max_{i=1}^n x_i=1\}$. This set is closed
under taking componentwise maxima and compact, hence $C$ contains
the greatest point $\max C$ with respect to the usual componentwise (partial) order in
$\Rpn$. Set $v:=\max C$ and let
$\alpha_t:=\max_{i=1}^n (A^{\otimes t}\otimes v)_i$ for any $t\geq 1$, then
$\max_{i=1}^n (\alpha_t^{-1}A^{\otimes t}\otimes v)_i=1$ and hence
$\alpha_t^{-1}A^{\otimes t}\otimes v\leq v$ by the definition of $v$. If
$\kappa>\alpha_t^{-1}$  then $\max_{i=1}^n (\kappa A^{\otimes t}\otimes v)_i=\kappa\alpha_t>1$ while
$\max_{i=1}^n v_i=1$ so
$\kappa A^{\otimes t}\otimes v\not\leq v$. This implies that
$\alpha_t^{-1}=[v: A^{\otimes t}\otimes v]$, and then
$\alpha_t^{-1}=[v:A\otimes v]^t=\alpha_1^{-t}$ by Lemma~\ref{l:shpiz}.

Consider the sequence of vectors $\{\alpha_1^{-t}A^{\otimes t}\otimes v\}_{t\geq 0}$,
which is the same as $\{\alpha_t^{-1}A^{\otimes t}\otimes v\}_{t\geq 0}$.
On one hand, the greatest component of each vector is $1$, by the above.
On the other hand, iterating $v\geq \alpha_1^{-1}A\otimes v$ we obtain that
$v\geq \alpha_1^{-1}A\otimes v\geq \alpha_1^{-2}A^{\otimes 2}\otimes v\geq\ldots$,
hence $\{\alpha_1^{-t}A^{\otimes t}\otimes v\}_{t\geq 0}$ has the limit, which we
denote by $u$. By the continuity $u$ satisfies $\alpha_1^{-1}A\otimes u=u$
and $\max_{i=1}^n u_i=1$, in particular $u\neq 0$. Since $C$ is closed, $u\in C\subseteq W$.
Hence $u$ is an eigenvector of $A$ belonging to $W$. The proof is complete.

\section{Acknowledgement}
The authors are grateful to Prof. Peter Butkovi\v{c} and Prof. Hans Schneider
for a number of important corrections and useful advice, and to the anonymous referee for constructive criticism.


\end{document}